\documentclass[a4paper,12pt,reqno]{amsart}
\usepackage{amsmath,amssymb}

\usepackage{graphicx}
\usepackage{amsthm}
\usepackage{thmtools}
\usepackage{comment}
\usepackage{hyperref}
\usepackage{float}
\usepackage{caption}
\usepackage{mathtools}
\usepackage{arcs}
\usepackage{yhmath}
\usepackage{tikz}
\usetikzlibrary{calc,intersections}
\makeatletter

\@namedef{subjclassname@2020}{%
\textup{2020} Mathematics Subject Classification} 
\makeatother

\newtheorem{theorem}{Theorem}
\newtheorem{lemma}[theorem]{Lemma}

\title{Weighted Generalizations of Zagier's Phenomenon}

\author{Dragomir Grozev}
\address{Institute of Mathematics and Informatics, Bulgarian Academy of Sciences, Acad. G. Bonchev 8,
1113 Sofia, Bulgaria}
\email{drago.grozev@gmail.com}

\author{Navid Safaei}
\address{Institute of Mathematics and Informatics, Bulgarian Academy of Sciences, Acad. G. Bonchev 8,
1113 Sofia, Bulgaria}

\email{navid@math.bas.bg}


\begin{document}
\begin{abstract}
We study sums associated with the action of $PGL_2(\mathbb Z)$
on continuous piecewise polynomial functions with exactly two
real roots, both irrational. We form weighted sums of the positive
parts of their normalized transforms, using nonnegative weights
compatible with translation, reflection, and inversion.

Under suitable continuity, finiteness, and convergence assumptions,
we prove that these sums are well defined, bounded, $1$-periodic,
and continuous on $\mathbb R$, and satisfy a reciprocal functional
equation. We also extend the construction to finite families
of distinct function orbits.

Our framework recovers Zagier's constancy result and includes
the full family of quadratic sums for which Bengoechea proved
convergence.
\end{abstract}


\maketitle
\vspace{-8mm}

\section{Introduction}
\label{sec:introduction}
We study weighted generalizations of Zagier's construction \cite{Zagier1999}. Let $P$ be a continuous piecewise polynomial with exactly two real roots, both irrational. We consider sums of the form
\[
T(x)=\sum_{\gamma\in\Gamma}
\lambda(\gamma,x)
\bigl(\varepsilon_\gamma P^\gamma(x)\bigr)_+,
\]
where $\Gamma=PGL_2(\mathbb Z)$ and the coefficients $\lambda(\gamma,x)$ satisfy compatibility conditions with its generators. We also consider sums indexed directly by $\Gamma_P\backslash\Gamma$.

Under continuity, convergence, and finiteness assumptions on the weights,
the sums define continuous periodic functions satisfying a reciprocal
functional equation. The aggregated setting recovers Zagier's constancy
result and proves convergence and continuity for Bengoechea's quadratic sums.

Related work includes the continued-fraction description of the
summands by Jameson~\cite{Jameson2016}, the identities involving
Hecke operators obtained by Jameson and Raji~\cite{JamesonRaji2013},
and the character-weighted sums studied by Wong~\cite{Wong2018}.
In another direction, Karabulut~\cite{Karabulut2022} proves analogous
results for sums of binary Hermitian forms. Here we allow compatible weights that depend on $x$.

Section~\ref{sec:setup} gives the two setups and their main results.
Section~\ref{sec:proofs} proves the functional equations, convergence,
and continuity. Section~\ref{sec:applications} treats finite families,
the applications to Zagier and Bengoechea, and weights depending on $x$.

\section{Setup and main results}
\label{sec:setup}
\subsection{Non-aggregated setting}
We assume that $P(x)$ is a function satisfying:
\begin{itemize}
\item $P:\mathbb R\to\mathbb R$ is continuous and piecewise
polynomial, with each polynomial piece of degree at most $2r$,
where $r\ge1$ is an integer.

\item $P$ has exactly two real roots $\alpha<\beta$, both irrational,
and
\[
P(x)>0 \quad \text{for } x\in(\alpha,\beta),
\qquad
P(x)<0 \quad \text{for } x\notin[\alpha,\beta].
\]

\item There exists $L<0$ such that
\[
\lim_{x\to-\infty}\frac{P(x)}{x^{2r}}
=
\lim_{x\to+\infty}\frac{P(x)}{x^{2r}}
=L.
\]
\end{itemize}

Let $\Gamma=PGL_2(\mathbb Z)$ act on $P$ by
\[
P^\gamma(x)
=
(cx+d)^{2r}P\left(\frac{ax+b}{cx+d}\right),
\qquad
\gamma=
\begin{pmatrix}
a&b\\
c&d
\end{pmatrix},
\]
with the value at a pole defined by continuity using the assumptions on $P$. For $\gamma\in \Gamma$, define 
\[
\varepsilon_{\gamma}:=
\begin{cases}
1,
&
P^\gamma(\infty)<0,
\\[2mm]
-1,
&
P^\gamma(\infty)>0.
\end{cases}
\]
Thus $\varepsilon_{\gamma}P^\gamma$ always has negative leading coefficient.
Let 
\[
\Gamma_P
=
\{\gamma\in\Gamma: \varepsilon_{\gamma}P^{\gamma}\equiv P\}.
\]

Then $\Gamma_P$ is a subgroup of $\Gamma$, with left coset space $\Gamma_P\backslash \Gamma$. The identity $\varepsilon_{\delta\gamma}P^{\delta\gamma}=\varepsilon_{\gamma}P^{\gamma}$ for $\delta\in \Gamma_P$ shows that $\varepsilon_\gamma P^\gamma$ depends only on $\Gamma_P\gamma$.

For a real-valued function $Q$, write
\[
Q_+(x):=\max\{Q(x),0\}.
\]

The weights $\lambda$ must be compatible with translation, reflection, and inversion. Write
\[
\mathsf T(x)=x+1, \qquad V(x)=-x, \qquad U(x)=\frac{1}{x}
\]
for these transformations.

We assume compatibility with $\sigma\in\{\mathsf T,V,U\}$:
\begin{equation}
\label{eq:eq_lambda_condition}
\lambda(\gamma\sigma^{-1},\sigma x)=\lambda(\gamma,x),
\quad \forall x,\sigma x\in\mathbb R.
\end{equation}
This is motivated by $(\gamma\sigma^{-1})(\sigma x)=\gamma x$.
We also impose the following properties:

\noindent\textbf{(1)}
For every $\gamma\in\Gamma$, the function
$x\mapsto\lambda(\gamma,x)\varepsilon_\gamma P^\gamma(x)$
is continuous.

\noindent\textbf{(2)}
For every $\gamma\in\Gamma$, the sum
$\sum_{\delta\in\Gamma_P}\lambda(\delta\gamma,x)|P^\gamma(x)|$
converges locally uniformly in $x$.

\noindent\textbf{(3)}
There are finitely many cosets $\Gamma_P\gamma$ such that
$\varepsilon_\gamma P^\gamma(0)>0$.

By translation and inversion, induction along the finite continued
fraction of a rational number extends Property~(3) to
\begin{equation}
\label{eq:eq_property_3}
\text{For every }x\in\mathbb Q,\qquad
\#\left\{
\Gamma_P\gamma\in\Gamma_P\backslash\Gamma:
\varepsilon_\gamma P^\gamma(x)>0
\right\}<\infty.
\end{equation}

\begin{theorem}[Non-aggregated setting]
\label{thm:main_thm}
Let $\lambda(\gamma,x)\ge 0$ and $P$ satisfy the preceding compatibility condition and Properties (1), (2), (3), together with the initial assumptions. Then  
  \begin{equation}
\label{eq:eq_def_sum}
T(x)=\sum_{\gamma\in\Gamma}\lambda(\gamma,x)
\bigl(\varepsilon_{\gamma}{P^\gamma}\bigr)_+(x),
\end{equation}  
is well defined, even, $1$-periodic, continuous and satisfies the corresponding functional equation:
\begin{equation}
\label{eq:eq_func_eq}
x^{2r}T\left(\frac{1}{x}\right)-T(x)=-R_\lambda(x), \quad x\ne 0,
\end{equation}
where
\begin{equation}
\label{eq:eq_R_definition}
R_\lambda(x)
:=
\sum_{\left(\varepsilon_{\gamma}{P^{\gamma}}\right)_+(0)>0} \lambda(\gamma, x) \varepsilon_{\gamma}P^{\gamma}(x).
\end{equation}
\end{theorem}

\subsection{Aggregated setting}
Since $\varepsilon_\gamma P^\gamma$ depends only on the coset,
set $P_{\overline\gamma}:=\varepsilon_\gamma P^\gamma$ and consider
\[
T(x)=
\sum_{\overline\gamma\in\Gamma_P\backslash\Gamma}
\lambda(\overline\gamma,x)(P_{\overline\gamma})_+(x),
\]
where $\lambda(\overline\gamma,x)\ge0$.
We write $\gamma$ for $\overline\gamma$ when unambiguous.
We assume compatibility with $\sigma\in\{\mathsf T,V,U\}$:
\begin{equation}
\label{eq:eq_coset_lambda_condition}
\lambda(\gamma\sigma^{-1},\sigma x)=\lambda(\gamma,x),
\quad \forall x,\sigma x\in\mathbb R.
\end{equation}
We also impose:

\noindent\textbf{(1')}
For every $\gamma\in\Gamma_P\backslash\Gamma$, the function
$x\mapsto\lambda(\gamma,x)P_\gamma(x)$ is continuous.

\noindent\textbf{(3')}
There are finitely many cosets $\gamma\in\Gamma_P\backslash\Gamma$
such that $P_\gamma(0)>0$.

There is no analogue of Property~(2), since there is no
stabilizer summation.

\begin{theorem}[Aggregated setting]
    \label{thm:main_coset_thm}
Let $\lambda(\overline{\gamma},x)\ge 0$,\quad $\overline{\gamma}\in \Gamma_P\backslash \Gamma$, satisfy the preceding compatibility condition and Properties (1'), (3'), with $P$ satisfying the initial assumptions. Then

\begin{equation}
\label{eq:eq_aggr_def_sum}
T(x)=\sum_{\overline{\gamma}\in \Gamma_P\backslash \Gamma}\lambda(\overline{\gamma},x)
\bigl(P_{\overline{\gamma}}\bigr)_+(x),
\end{equation}
is well defined, even, $1$-periodic, and continuous and satisfies the corresponding reciprocal functional equation:
\begin{equation*}
x^{2r}T\left(\frac{1}{x}\right)-T(x)=-R_\lambda(x), \quad x\ne 0,
\end{equation*}
where
\begin{equation*}
R_\lambda(x)
:=
\sum_{\substack{\overline{\gamma}\in\Gamma_P\backslash\Gamma\\ P_{\overline{\gamma}}(0)>0}} \lambda(\overline{\gamma}, x) P_{\overline{\gamma}}(x).
\end{equation*}

\end{theorem}

\section{Proofs}
\label{sec:proofs}

\subsection{Functional equations}
\label{subsec:functional-equations}

For \(x\in\mathbb{Q}\), \eqref{eq:eq_property_3} gives finitely many cosets in
$\Gamma_P\backslash\Gamma$ contributing at $x$. Property~(2) gives
convergence within each, so $T(x)$ is well defined for \(x\in\mathbb{Q}\).

\subsubsection*{Translation} 
Set $\sigma:=\mathsf T$. We have
\begin{align*}
    T(x+1)&=\sum_{\gamma\in\Gamma}\lambda(\gamma,x+1)
\bigl(\varepsilon_{\gamma}{P^\gamma}\bigr)_+(x+1)\\
&= \sum_{\gamma\in\Gamma}\lambda(\gamma,\sigma x)
\bigl(\varepsilon_{\gamma}P^{\gamma\sigma}\bigr)_+(x)\\
&=\sum_{\gamma\sigma\in\Gamma}\lambda(\gamma\sigma,x)
\bigl(\varepsilon_{\gamma \sigma}P^{\gamma\sigma}\bigr)_+(x)\\
&=T(x).
\end{align*}
Hence $T$ is $1$-periodic.

\subsubsection*{Reflection}
The same calculation with $\sigma=V$ gives $T(-x)=T(x)$.

\subsubsection*{Inversion}

For $U: x\mapsto 1/x$, set $\sigma:=U$ and write $T=T_0+T_1$, where
\[
T_0(x):= \sum_{\gamma\in \Gamma,\, \left(\varepsilon_{\gamma}P^{\gamma}\right)_+(0)>0} \lambda(\gamma,x) \left(\varepsilon_{\gamma}P^{\gamma}\right)_+(x).
\]
\[
T_1(x):= \sum_{\gamma\in \Gamma,\, \left(\varepsilon_{\gamma}P^{\gamma}\right)_+(0)=0} \lambda(\gamma,x) \left(\varepsilon_{\gamma}P^{\gamma}\right)_+(x).
\]
Assuming $x\ne 0$, we have
\begin{align*}
x^{2r}T_1\left(\frac{1}{x}\right)
&=
\sum_{\left(\varepsilon_\gamma P^\gamma\right)_+(0)=0}
\lambda(\gamma,\sigma x)\,
x^{2r}\left(\varepsilon_\gamma P^\gamma\right)_+(\sigma x)
\\
&=
\sum_{\left(\varepsilon_\gamma P^\gamma\right)_+(0)=0}
\lambda(\gamma,\sigma x)\,
\left(\varepsilon_{\gamma\sigma}P^{\gamma\sigma}\right)_+(x)
\\
&=
\sum_{\left(\varepsilon_{\gamma\sigma}P^{\gamma\sigma}\right)_+(0)=0}
\lambda(\gamma\sigma,x)\,
\left(\varepsilon_{\gamma\sigma}P^{\gamma\sigma}\right)_+(x)
\\
&=
T_1(x).
\end{align*}
This implies
\[
x^{2r}T_1(1/x)=T_1(x).
\]

Similarly,
\begin{align*}
    x^{2r}T_0\left(\frac{1}{x}\right)&=\sum_{\left(\varepsilon_{\gamma}{P^{\gamma}}\right)_+(0)>0} \lambda(\gamma,\sigma x) x^{2r} \left(\varepsilon_{\gamma}{P^{\gamma}}\right)_+(\sigma x)\\
    &=\sum_{\left(\varepsilon_{\gamma}{P^{\gamma}}\right)_+(0)>0} \lambda(\gamma,\sigma x) \left(-\varepsilon_{\gamma\sigma}P^{\gamma\sigma}\right)_+(x)\\ 
    &=\sum_{\left(\varepsilon_{\gamma}{P^{\gamma}}\right)_+(0)>0} \lambda(\gamma,\sigma x) \left({-\varepsilon_{\gamma\sigma}P^{\gamma\sigma}}(x) + \left(\varepsilon_{\gamma\sigma} P^{\gamma\sigma}\right)_+(x) \right)\\
    &=\sum_{\left(\varepsilon_{\gamma\sigma}{P^{\gamma\sigma}}\right)_+(0)>0} \lambda(\gamma\sigma, x) \left({-\varepsilon_{\gamma\sigma}P^{\gamma\sigma}}(x) + \left(\varepsilon_{\gamma\sigma}{P^{\gamma\sigma}}\right)_+(x) \right).
\end{align*}
Thus

\[
x^{2r}T_0(1/x)-T_0(x)
=-\sum_{\left(\varepsilon_{\gamma}P^{\gamma}\right)_+(0)>0} \lambda(\gamma, x) \varepsilon_{\gamma}P^{\gamma}(x) .
\]
Together with $x^{2r}T_1(1/x)=T_1(x)$ and
\eqref{eq:eq_R_definition}, this gives \eqref{eq:eq_func_eq}.

\subsection{Convergence and continuity}
\label{subsec:convergence-continuity}
By \eqref{eq:eq_def_sum} and the preceding subsection, $T$ is $1$-periodic and satisfies \eqref{eq:eq_func_eq} wherever defined. Since $T(x)$ is defined for $x\in\mathbb{Q}$, we first bound $T$ on the rationals, extend convergence to $\mathbb{R}$, and then prove continuity of $T$. 

\begin{lemma}
\label{lem:lem_1}
    The function $R_{\lambda}(x)$ is continuous on $\mathbb{R}$ and 
    \[
    \displaystyle \lim_{x\to\pm\infty} R_{\lambda}(x)/x^{2r} = -T(0).
    \]
\end{lemma}
\begin{proof}

In \eqref{eq:eq_R_definition}, $\gamma$ runs over finitely many cosets by Property~(3). Property~(1) makes $\lambda(\gamma,x)\varepsilon_{\gamma}P^{\gamma}$ continuous, and Property~(2) allows locally uniform summation within each coset. Thus $R_\lambda(x)$ is continuous.

For the asymptotic, put $t=1/y$. If
\[
\left(\varepsilon_\gamma P^\gamma\right)_+(0)>0,
\]
then $\varepsilon_{\gamma U}=-\varepsilon_\gamma$, and therefore
\[
\varepsilon_{\gamma U}P^{\gamma U}(t)
=
-t^{2r}\varepsilon_\gamma P^\gamma(1/t).
\]
Using the compatibility condition for $\lambda$, we obtain
\[
\frac{R_\lambda(y)}{y^{2r}}
=
-\sum_{\left(\varepsilon_\gamma P^\gamma\right)_+(0)>0}
\lambda(\gamma U,t)\,
\varepsilon_{\gamma U}P^{\gamma U}(t).
\]
The map $\gamma\mapsto\gamma U$ is a bijection of the set of indices
satisfying
\[
\left(\varepsilon_\gamma P^\gamma\right)_+(0)>0
\]
onto itself, since
\[
\varepsilon_{\gamma U}P^{\gamma U}(0)
=
-\varepsilon_\gamma P^\gamma(\infty)>0.
\]
Hence
\[
\frac{R_\lambda(y)}{y^{2r}}
=
-\sum_{\left(\varepsilon_\gamma P^\gamma\right)_+(0)>0}
\lambda(\gamma,t)\varepsilon_\gamma P^\gamma(t).
\]
Letting $y\to\pm\infty$, equivalently $t\to0^\pm$, and using local
uniform convergence gives
\[
\lim_{y\to\pm\infty}\frac{R_\lambda(y)}{y^{2r}}
=
-\sum_{\left(\varepsilon_\gamma P^\gamma\right)_+(0)>0}
\lambda(\gamma,0)\varepsilon_\gamma P^\gamma(0)
=
-T(0).
\]
\end{proof}

\noindent \textbf{Remark.} The same proof applies to the direct aggregated setting: all sums are taken over \(\Gamma_P\backslash\Gamma\); Property~$(3')$ makes the
remainder sum finite, and Property~$(1')$ makes each of its summands
continuous.

We use standard facts about continued fractions [see \cite[Chapter~X]{HW}]. By periodicity of $T$, take $x\in[0,1)$ and write the continued fraction of $x$ as
\[
x=[0;n_1,n_2,\ldots]
\]
where it is finite when $x\in\mathbb{Q}$. We denote its successive
tails by
\[
\alpha_0=x,\qquad
\alpha_{j-1}=\frac{1}{n_j+\alpha_j}.
\]
For a rational $x=[0;n_1,\ldots,n_k]$ we take $\alpha_k=0$. Let us set
\[
y_j:=n_j+\alpha_j=\frac{1}{\alpha_{j-1}}.
\]
Thus
\[
\alpha_{j-1}=\frac1{y_j},
\qquad
y_j\geq1.
\]

\noindent\textbf{Uniform boundedness.} For $x\in\mathbb{Q}$ with $x=[0;n_1,\ldots,n_k]$, iteration gives
\[
T(x)
=
T(0)\prod_{i=1}^{k}\frac1{y_i^{2r}}
-
\sum_{j=1}^{k}
\frac{R_\lambda(y_j)}{y_j^{2r}}
\prod_{i=1}^{j-1}\frac1{y_i^{2r}},
\]
together with
\[
\sup_{y\geq1}
\left|\frac{R_\lambda(y)}{y^{2r}}\right|<\infty
\]
and the two-step contraction
\begin{equation}
\label{eq:eq_two_step_contr}
\frac1{y_jy_{j+1}}\leq\frac23,
\end{equation}
we obtain
\[
\sup_{x\in\mathbb Q}|T(x)|\leq C
\]
for some constant $C$. We now prove convergence for every real $x$, following the argument used in \cite[p.~1158]{Zagier1999}. Fix $x_0\in\mathbb R$ and take an arbitrary finite set
$\mathcal F\subset\Gamma$, such that
$\varepsilon_{\gamma}P^{\gamma}(x_0)>0$ for all
$\gamma\in\mathcal F$. Consider
\[
T_{\mathcal{F}}(x):= \sum_{\gamma\in \mathcal{F}} \lambda(\gamma,x) \left(\varepsilon_{\gamma} P^{\gamma} \right)_+(x).
\]
The functions $\varepsilon_\gamma P^\gamma$,
$\gamma\in\mathcal F$, are positive in a neighborhood of $x_0$.
Hence Property~(1) implies that $T_{\mathcal F}$ is continuous
at $x_0$.

For any $y\in\mathbb Q$,
\[
T_{\mathcal F}(y)\leq T(y)\leq C.
\]
Letting $y\to x_0$ through rational values gives $T_{\mathcal F}(x_0)\leq C$. Since $\mathcal F$ is arbitrary,
\[
T(x_0)=\sup_{\mathrm{finite}\, \mathcal F\subset\Gamma\ }
T_{\mathcal F}(x_0)\leq C.
\]
Thus the nonnegative series \eqref{eq:eq_def_sum} converges
for every $x\in\mathbb R$, and $T$ is uniformly bounded.

\noindent \textbf{Continuity.} To prove that $T$ is continuous, periodicity allows us to restrict attention to
$x_0\in[0,1]$.

Suppose first that $x_0$ is irrational, with continued fraction
\[
x_0=[0;n_1,n_2,\ldots].
\]
Fix $k$. If $x$ is sufficiently close to $x_0$, then $x$ has the same
first $k$ partial quotients:
\[
x=[0;n_1,\ldots,n_k,\ldots].
\]
Writing
\[
y_j=n_j+\alpha_j,
\qquad
\alpha_{j-1}=\frac1{y_j},
\]
and iterating the functional equation $k$ times, we obtain
\begin{equation}
\label{eq:eq_recurence_T}
T(x)=T(\alpha_k(x))
\prod_{i=1}^{k}\frac1{y_i(x)^{2r}}-\sum_{j=1}^{k} R_\lambda(y_j(x)) \prod_{i=1}^{j}\frac1{y_i(x)^{2r}},
\end{equation}
and the analogous formula for $x_0$.

From~\eqref{eq:eq_two_step_contr}, we get 
\[
\prod_{i=1}^{k}\frac1{y_i^{2r}}
\ll
\left(\frac23\right)^{2r\lfloor k/2\rfloor}.
\]
Since $T$ is bounded on $[0,1]$, the first term in \eqref{eq:eq_recurence_T} can therefore
be made arbitrarily small by choosing $k$ sufficiently large.

For this fixed $k$, the quantities $y_j(x)$, $j\leq k$, tend to
$y_j(x_0)$ as $x\to x_0$. Since $R_\lambda$ is continuous, the finite
sum in \eqref{eq:eq_recurence_T} tends to the corresponding finite sum for $x_0$. It follows
that $T(x)\longrightarrow T(x_0)$.

Now suppose that $x_0$ is rational. We consider first the case $x_0=0$.
Since $R_\lambda(0)=T(0)$, the functional equation and boundedness give
\[
|T(x)-T(0)|
\le C|x|^{2r}
   +|R_\lambda(x)-R_\lambda(0)|
\longrightarrow 0
\qquad (x\to 0),
\]
which proves the continuity at $x=0$. By $1$-periodicity, continuity at $0$ also implies continuity at $1$.

Assume now $x_0\ne 0,1$. Let us write its standard finite
continued fraction as
\[
x_0=[0;n_1,\ldots,n_k],
\qquad n_k>1.
\]
For $x$ sufficiently close to $x_0$, the first $k-1$ partial quotients
agree with those of $x_0$. After this common prefix, depending on the
side from which $x$ approaches $x_0$, the continued fraction has one
of the forms
\[
[0;n_1,\ldots, n_{k-1}, n_k,N,\ldots]
\qquad\text{or}\qquad
[0;n_1, \ldots, n_{k-1}, n_k-1,1,N,\ldots],
\]
where $N\to\infty$ as $(x\to x_0)$. In the first continued-fraction form, the remaining tail after
$n_k$ tends to $0$. In the second form, the corresponding tail
tends to $1$; by periodicity and evenness,
\[
T(1-z)=T(z),
\]
so its value again tends to $T(0)$. Substitution into the finite
recurrence~\eqref{eq:eq_recurence_T}, together with the continuity
of $R_\lambda$, gives
\[
T(x)\longrightarrow T(x_0).
\]
Thus $T$ is continuous at every real point.

\smallskip
\noindent\textbf{Remark.}
The continued-fraction expansion in
\eqref{eq:eq_recurence_T} converges exponentially.
Indeed, boundedness of $T$ and
\eqref{eq:eq_two_step_contr} give
\[
\left|
T(x)+\sum_{j=1}^{N}R_\lambda(y_j(x))
\prod_{i=1}^{j}y_i(x)^{-2r}
\right|
\le C\left(\frac23\right)^{2r\lfloor N/2\rfloor},
\]
for every irrational $x\in(0,1)$ and $N\ge1$,
with $C$ independent of $x$ and $N$.

\section{Applications and examples}
\label{sec:applications}

\subsection*{Finite families}
For a single function $P$ satisfying our assumptions, write
\[
T_{\lambda,P}(x)= \sum_{\gamma\in \Gamma} \lambda_P(\gamma,x)
\bigl(\varepsilon_{\gamma,P}{P^\gamma}\bigr)_+(x).
\]
Take a finite family $\mathcal{P}$ satisfying Section~\ref{sec:setup}, with distinct orbits:
\[
\varepsilon_{\gamma_1,P_1}P_1^{\gamma_1} \neq \varepsilon_{\gamma_2,P_2}P_2^{\gamma_2}, \text{ if } P_1\ne P_2\in \mathcal{P},\, \forall \gamma_1,\gamma_2\in \Gamma.
\]
With admissible weights $\lambda_P(\gamma,x)$ for each $P\in\mathcal{P}$, set
\[
T(x):=\sum_{P\in\mathcal{P}} T_{\lambda,P}(x).
\]
This finite sum is $1$-periodic and continuous. The same applies to aggregated sums. 

\subsection*{Zagier's result}
Let $\mathcal{F}$ consist of integral quadratics with negative leading coefficient and fixed nonsquare discriminant $D>0$.   
Zagier \cite{Zagier1999} proved
\[
T(x):= \sum_{P\in \mathcal{F}} P_+(x) =\mathrm{const}.
\]
By the classical reduction theory of indefinite binary quadratic
forms (recalled in \cite[Section~2]{Bengoechea2015}), there are
finitely many equivalence classes of discriminant $D$.
Hence we may choose a finite family $\mathcal{P}\subset\mathcal{F}$
containing exactly one representative of each normalized orbit,
so that
\[
\mathcal{F}=\bigcup_{P\in\mathcal{P}}
\left\{\varepsilon_{\gamma,P}P^\gamma:\gamma\in\Gamma\right\}.
\]
Then
\[
T(x)=\sum_{P\in\mathcal{P}}\, \sum_{\gamma\in \Gamma_P\backslash \Gamma} \lambda_P(\gamma, x)\left(\varepsilon_{\gamma, P}P^{\gamma}\right)_+(x),
\]
where $\lambda_P(\gamma, x):=1$ satisfies our conditions; finiteness follows from integrality and the fixed discriminant. Indeed, the relevant quadratics $aX^2+bX+c$ satisfy $a<0<c$ and $b^2+4|a|c=D$, allowing only finitely many
integer triples. Moreover,
\[
R_{\lambda}:= \sum_{P\in\mathcal{P}} R_{\lambda, P}
\]
is an even quadratic, as are the $R_{\lambda, P}$. Write $R_{\lambda}=a x^2+c$. Lemma~\ref{lem:lem_1} gives
$a=-T(0)$, while $c=R_\lambda(0)=T(0)$; hence $c+a=0$. Set $T^*(x):=T(x)+a$. Then using \eqref{eq:eq_func_eq} we get
\[
x^2T^*\left(\frac{1}{x}\right)-T^*(x)=x^2T\left(\frac{1}{x}\right)-T(x)+ax^2+c=0.
\]
Since $T^*(0)=T(0)+a=c+a=0$, iteration of the equation for $T^*$ along a finite continued fraction gives $T^*(x)=0$ for $x\in\mathbb{Q}$. Continuity yields $T(x)=-a, \ \forall x\in \mathbb{R}$.

We may also assign constant weights to each $P\in \mathcal{P}$: 
\[
\lambda_P(\gamma, x)= \lambda_P>0,\quad \forall \gamma \in \Gamma_P\backslash \Gamma.
\]
Again $T$ is constant, since $R_{\lambda}$ is an even quadratic.

\subsection*{Bengoechea's result}

Let $D>0$ be a nonsquare discriminant and let $m\ge 1$ be an integer.
Bengoechea \cite{Bengoechea2015} considers the sums
\[
A_{m+1,D}(x)
=
\sum_{\substack{R\in\mathbb Z[X],\ \deg R=2\\
\operatorname{disc}(R)=D,\ R(\infty)<0<R(x)}}
R(x)^m.
\]

Let
\[
S_D^{\mathrm{Sim}}
=
\left\{
S\in\mathbb Z[X]:
\deg S=2,\ \operatorname{disc}(S)=D,\ 
S(0)<0<S(\infty)
\right\},
\]
and let $\Gamma(x)$ denote the set of Möbius transformations
arising from the regular continued fraction of $x$.

A pair
\[
(S,\gamma)\in S_D^{\mathrm{Sim}}\times\Gamma(x)
\]
is admissible at $x$ if
\[
S(\gamma(\infty))<0<S(\lfloor\gamma(x)\rfloor),
\]
with the usual terminal convention when $\gamma(x)=\infty$.
Bengoechea's bijection sends each admissible pair to $S^\gamma$,
and gives the representation
\[
A_{m+1,D}(x)
=
\sum_{\substack{S\in S_D^{\mathrm{Sim}},\ \gamma\in\Gamma(x)\\
(S,\gamma)\text{ admissible at }x}}
\bigl(S^\gamma(x)\bigr)^m.
\]
See \cite[Theorem~3.1 and Corollary~3.4]{Bengoechea2015}.

For the normalized orbits of integral quadratics of discriminant $D$, choose representatives
\[
Q_1,\ldots,Q_h
\]
with negative leading coefficient, and set
\[
P_j=
\begin{cases}
Q_j^m,&m\text{ odd},\\
Q_j|Q_j|^{m-1},&m\text{ even}.
\end{cases}
\]
It suffices to check one $P=P_j$, with $Q=Q_j$ and $r=m$.

The function $P$ is continuous and piecewise polynomial, with the roots and signs of $Q$. If $a<0$ is the leading coefficient of $Q$, then
\[
\lim_{x\to-\infty}\frac{P(x)}{x^{2m}}
=
\lim_{x\to+\infty}\frac{P(x)}{x^{2m}}
=
a|a|^{m-1}<0.
\]
Thus $P$ satisfies our initial assumptions.

Using the quadratic action on $Q$ and the degree-$2m$
action on $P$, we have
\[
P^\gamma
=
Q^\gamma|Q^\gamma|^{m-1}.
\]
Since $u\mapsto u|u|^{m-1}$ preserves signs and is injective,
$\varepsilon_{\gamma,P}=\varepsilon_{\gamma,Q}$ and $\Gamma_P=\Gamma_Q$. For $\overline\gamma=\Gamma_P\gamma$, write $P_{\overline\gamma}=\varepsilon_\gamma P^\gamma$.

For the normalized transform $R$ of $Q$,
\[
P_{\overline\gamma}=R|R|^{m-1},
\qquad
(P_{\overline\gamma})_+=(R_+)^m.
\]

Define the coefficients directly on the coset space by
\[
\lambda(\overline\gamma,x)=1,
\qquad
\overline\gamma\in\Gamma_P\backslash\Gamma,\quad
x\in\mathbb R.
\]
The associated function is
\[
T_P(x)
=
\sum_{\overline\gamma\in\Gamma_P\backslash\Gamma}
(P_{\overline\gamma})_+(x).
\]
Each normalized transform of $Q$ contributes its $m$-th power when positive at $x$.

Compatibility is immediate since $\lambda=1$.
Property~(1') follows from the continuity of
$P_{\overline\gamma}$.
Since $P_{\overline\gamma}(0)>0$ if and only if $R(0)>0$,
Property~(3') follows from the finiteness argument
in the Zagier case.

Thus $T_P$ is well defined, bounded, even, $1$-periodic, and continuous, and satisfies the reciprocal relation. Finally,
\[
A_{m+1,D}(x)
=
\sum_{j=1}^{h}T_{P_j}(x).
\]
Hence the same conclusions hold for the full Bengoechea sum.

\subsection*{An example with coefficients depending on $x$} 

For weights $\lambda$ depending on $x$, take an integral quadratic $P$ satisfying the initial assumptions and a function $f$ specified below. Set
\[
\lambda(\gamma,x):= f(\gamma(x)),
\]
where $f(\infty)=0$. Compatibility~\eqref{eq:eq_lambda_condition} is automatic. Choose a nonnegative continuous $f$, positive only in $(\alpha_1, \beta_1)$, with $\alpha<\alpha_1<\beta_1<\beta$. Continuity of $f$ and $P^{\gamma}$ gives Property~(1): near a pole of $\gamma$, the factor $f(\gamma(x))$ vanishes because $f$ has bounded support. Property~(3) holds for these quadratics. 

Let us verify Property~(2). Consider the subsets of $\Gamma_P$:
\[
H_0:=\{\delta\in \Gamma_P : \delta(\alpha)=\alpha, \delta(\beta)=\beta \}\quad H_1:=\{\delta\in \Gamma_P : \delta(\alpha)=\beta, \delta(\beta)=\alpha \}.
\]
Here $H_0$ is finite or infinite cyclic in $\Gamma_P$, and $H_1$ is empty or $H_1=H_0 h_1$ for $h_1\in H_1$. If $H_0$ is finite, Property (2) is immediate. Otherwise, write $H_0=\langle h \rangle \cong \mathbb{Z}$ and set 
\[
S_x:= \left\{n\in \mathbb{Z} : h^{n}(x)\in (\alpha_1,\beta_1) \right\}.
\]
Since $[\alpha_1,\beta_1]\subset (\alpha,\beta)$, we have $\#S_x\le C(\alpha_1,\beta_1)$ with $C$ independent of $x$. For fixed $\gamma\in\Gamma$ and compact $K$ avoiding $\gamma^{-1}\left( \{\alpha, \beta\}\right)$,
\[
\#\bigcup_{x\in K} S_{\gamma(x)}<\infty.
\]
Indeed, as $n\to+\infty$ and $n\to-\infty$, $h^n(\gamma(x))$ tends uniformly for $x\in K$ to the two fixed points $\alpha,\beta$, in some order. Since $[\alpha_1,\beta_1]$ avoids both fixed points, only finitely many integers $n$ can occur.

Fix $\gamma\in \Gamma$ and let $\varepsilon>0$. There exists a neighborhood $U$ of $\left\{ \gamma^{-1}(\alpha)\right\}\cup \left\{ \gamma^{-1}(\beta)\right\}$ such that
\[
\left|P^{\gamma}(x)\right| \sum_{\delta\in \Gamma_P} \lambda(\delta\gamma, x)\le  2C(\alpha_1,\beta_1)\|f\|_{\infty}\cdot \left|P^{\gamma}(x)\right|<\varepsilon, \quad \forall x\in U. 
\]
For compact $K$ with $K\cap U=\emptyset$, the same argument gives a finite $F\subset \Gamma_P $ such that
\[
\left|P^{\gamma}(x)\right| \sum_{\delta\in \Gamma_P} \lambda(\delta\gamma, x)=
\left|P^{\gamma}(x)\right| \sum_{\delta\in F} \lambda(\delta\gamma, x),\quad \forall x\in K.
\]
These estimates give local uniform convergence. Here $T$ need not be constant.

\bibliographystyle{plain}
\bibliography{references}

\end{document}